\documentclass[12pt,twoside]{amsart}
\usepackage{amsmath, amsthm, amscd, amsfonts, amssymb, graphicx}
\usepackage{enumerate}
\usepackage[colorlinks=true,
linkcolor=blue,
urlcolor=cyan,
citecolor=red]{hyperref}
\usepackage{mathrsfs}
\newtheorem{theorem}{Theorem}[section]
\newtheorem{lemma}{Lemma}[section]

\newtheorem{corollary}{Corollary}[section]

\newtheorem{proposition}{Proposition}[section]
\numberwithin{equation}{section}

\begin{document}
\title{Ando-Hiai and Golden-Thomspon inequalities}
\author{Mohammed Sababheh and Hamid Reza Moradi}
\subjclass[2010]{Primary 47A63, Secondary 46L05, 47A60.}
\keywords{Ando-Hiai inequality, Golden-Thompson inequality, P\'olya-Szeg\"o inequality, operator monotone function} \maketitle

\begin{abstract}
The original Ando-Hiai and Golden-Thompson inequalities present comparisons for the operator geometric mean $\sharp_v$ when $0\leq v\leq 1.$ Our main target in this article is to study these celebrated inequalities for means other than the geometric mean and for the geometric mean when $v\not\in [0,1].$
\end{abstract}
\pagestyle{myheadings}
\markboth{\centerline {Ando-Hiai and Golden-Thompson inequalities}}
{\centerline {M. Sababheh \& H.R. Moradi}}
\bigskip
\bigskip
\section{Introduction}
For two complex $n\times n$ Hermitian matrices $A$ and $B$, the Golden-Thompson inequality states that \cite{golden,thompson}
\begin{equation}\label{ineq_original_GT_Tr}
{\text{tr}}\left(e^{A+B}\right)\leq {\text{tr}}\left(e^{A}e^{B}\right),
\end{equation}
as a non-commutative version of the scalar identity $e^{a+b}=e^{a}e^{b}.$ The inequality \eqref{ineq_original_GT_Tr} has its application in statistical mechanics and random matrix theory. In an unpublished work,  Dyson proved \eqref{ineq_original_GT_Tr} when he was studying random matrix theory and its application to nuclear physics. Thus, this inequality is expected to have future application in this direction. Further applications of this inequality can be found in \cite{survey}.

The inequality \eqref{ineq_original_GT_Tr} has been extended in various forms, among which we are interested in the following unitarily invariant norm version \cite{ando1}
\begin{equation}\label{ineq_GT_ando_hiai}
\left\|\left(e^{pA}\sharp_ve^{pB}\right)^{\frac{1}{p}}\right\|\leq \left\|e^{A\nabla_vB}\right\|,\;0\leq v\leq 1,\;p>0
\end{equation}
where $A,B$ are Hermitian, $\sharp_v$ is the geometric mean, $\nabla_v$ is the arithmetic mean and $\|\;\;\|$ is any unitarily invariant norm.

The inequality \eqref{ineq_GT_ando_hiai} has been reversed in \cite[Theorem 3.4]{seo} using the Specht ratio. Very recently, the authors in \cite[Corollary 2.7]{gumus} have shown  a stronger reverse. However, all these results treat the geometric mean for $0\leq v\leq 1.$

One target of the current paper is to prove variants of \eqref{ineq_GT_ando_hiai}. These new versions extend the domain of $v$ to values outside the interval $[0,1]$ and also extend the treatment of the geometric mean $\sharp_v$ to any operator mean between the harmonic mean $!_v$ and the arithmetic mean $\nabla_v$. However, these extensions will be at the cost of an additional constant. For example, we prove that, under mild conditions on $A,B$,
 $$\left\|\left(e^{pA}\sharp_v e^{pB}\right)^{\frac{1}{p}}\right\|\leq \gamma_p^{\frac{1}{p}}\left\|e^{A\nabla_vB}\right\|,\;v>1$$ where $\gamma_p$ is a certain constant. Moreover, we prove that, for some constant $L$, depending on $p$,
 $$\left\|\left(e^{pA}\sigma_v e^{pB}\right)^{\frac{1}{p}}\right\|\leq L^{\frac{1}{p}}\left\|e^{A\nabla_vB}\right\|,\;0\leq v\leq 1,$$ for any operator mean $\sigma_v$ between $!_v$ and $\nabla_v$. To the best of our knowledge, such extensions have not been considered earlier in the literature.

Our methods for proving such results allow us, also, to obtain variants of the well known Ando-Hiai inequality, which 
 asserts that if $A$ and $B$ are positive operators, then for any $v\in \left[ 0,1 \right]$, we have \cite{ando1}
\begin{equation}\label{8}
\left\| {{A}^{p}}{{\sharp}_{v}}{{B}^{p}} \right\|\le {{\left\| A{{\sharp}_{v}}B \right\|}^{p}}\quad\text{ for all }p>1,
\end{equation}
or equivalently,
\[A{{\sharp}_{v}}B\le I\quad\text{ }\Rightarrow \quad\text{ }{{A}^{p}}{{\sharp}_{v}}{{B}^{p}}\le I\quad\text{ for all }p>1.\]
A counterpart to the Ando--Hiai inequality \eqref{8} has been presented by Nakamoto and
Seo \cite{nakamoto} as follows
\begin{equation}\label{11}
{{\left\| A{{\sharp}_{v}}B \right\|}^{p}}\le \frac{1}{K\left( {{h}^{2p}},v \right)}\left\| {{A}^{p}}{{\sharp}_{v}}{{B}^{p}} \right\|\quad\text{ for all }p>1, 0\leq v\leq 1
\end{equation}
whenever $mI\le A,B\le MI$ for some scalars $0<m<M$, $h=\frac{M}{m}$, and 
\[K\left( h,v \right)\equiv \frac{{{h}^{v}}-h}{\left( v-1 \right)\left( h-1 \right)}{{\left( \frac{v-1}{v}\frac{{{h}^{v}}-1}{{{h}^{v}}-h} \right)}^{v}}\] is a generalized Kantorovich constant.

Our second target in this article is to extend the domain of the Ando-Hiai inequality to $v\not\in [0,1]$ and to extend it to arbitrary means, rather than the geometric mean. For example, we show that when $0<mI\leq A,B\leq MI$ and $\sigma_v,\tau_v$ are arbitrary operator means between the harmonic and arithmetic means,
$$\left\| A{{\sigma }_{v}}B \right\|^{p}\le L\left( {{m}^{p}},{{M}^{p}} \right)\left\| {{A}^{p}}{{\tau }_{v}}{{B}^{p}} \right\|\quad\text{ for all }p>1$$ where $L\left( {{m}^{p}},{{M}^{p}}\right)$ is a constant independent of $\sigma_v$ and $\tau_v$.

The above extensions we prove will follow as special cases of a more general treatment of operator monotone functions. This treatment, of operator monotone functions, will imply the P\'{o}lya-type inequality
$$f\left( \Phi \left( A{{\sigma }_{v}}B \right) \right)\le \xi \psi \left( f\left( \Phi \left( A \right) \right){{\tau }_{v}}f\left( \Phi \left( B \right) \right) \right),$$ where $f$ is an operator monotone function, $A,B$ are positive operators related via a sandwich condition, $\xi,\psi$ are certain constants and $\Phi$ is a normalized positive linear map.

In the sequel, $\mathcal{B}(\mathcal{H})$ will denote the algebra of bounded linear operators on a complex Hilbert space $\mathcal{H}$, while $\mathcal{M}_n$ will stand for the algebra of $n\times n$ complex matrices. Our Golden-Thompson inequalities will be valid for matrices in $\mathcal{M}_n$, while all other results will be valid for operators in $\mathcal{B}(\mathcal{H}).$ The letters $m,M$ and their indices will be used for positive real numbers, while $A,B$ will stand for Hilbert space operators (or matrices in the case of the Golden-Thompson inequality). Further, all monotone functions $f$ we deal with will be of the form $f:(0,\infty)\to (0,\infty).$
\section{Extending the domain of the  Ando-Hiai and Golden-Thompson inequalities}
In this section, we present Ando-Hiai and  Golden-Thompson inequalities for $v\not\in[0,1].$ 

The following inequalities were pointed out in \cite{1}. For completeness, we present a simple proof.
 
 \begin{lemma}\label{lemma_amgm_reversed}
  Let $A,B\in\mathcal{B}(\mathcal{H})$ be such that $0<m_2I\le A\le m_1I<M_1I\le B\le M_2I$ and $v\not\in[0,1].$ Then
		\[\frac{m_1{{\sharp}_{v}}M_1}{m_1{{\nabla }_{v}}M_1}A{{\nabla }_{v}}B\le A{{\sharp}_{v}}B\le \frac{m_2{{\sharp}_{v}}M_2}{m_2{{\nabla }_{v}}M_2}A{{\nabla }_{v}}B\]
		and
		\[\frac{m_1{{!}_{v}}M_1}{m_1{{\sharp}_{v}}M_1}A{{\sharp}_{v}}B\le A{{!}_{v}}B\le \frac{m_2{{!}_{v}}M_2}{m_2{{\sharp}_{v}}M_2}A{{\sharp}_{v}}B.\]
\end{lemma} 
\begin{proof}
Let, for $v\not\in[0,1]$, $$f(x)=\frac{1-v+v\;x}{x^v},\;1<\frac{M_1}{m_1}\leq x\leq \frac{M_2}{m_2}.$$ Then $$f'(x)=v(1-v)(x-1)x^{1+v}.$$ That is, $f$ is decreasing for $x>1$, since $v\not\in[0,1].$ Therefore,
$$\frac{M_1}{m_1}\leq x\leq \frac{M_2}{m_2}\Rightarrow f\left(\frac{M_1}{m_1}\right)\leq f(x)\leq f\left(\frac{M_2}{m_2}\right).$$
Then the first two desired inequalities follow by applying a standard functional calculus argument using $x=A^{-\frac{1}{2}}BA^{-\frac{1}{2}}.$ The other inequalities involving the harmonic mean follow similarly.
\end{proof}

As a corollary, we have the following auxiliary inequality that we will use to prove the desired Ando-Hiai and Golden-Thompson inequalities. Notice first that when $f:(0,\infty)\to (0,\infty)$ is operator monotone, it is operator concave \cite{ando}. Therefore, when $v\not\in [0,1],$ we have \cite{sab_laaa}
$$f(A\nabla_v B)\leq f(A)\nabla_vf(B).$$ Further, the arithmetic-geometric mean inequality states that for $0\leq v\leq 1,$ one has $A\sharp_v B\leq A\nabla_v B$. However, when $v\not\in[0,1],$ the inequality is reversed.\\
We remark here that our functions $f$ will be defined on $(0,\infty).$ This is the main reason we consider $v\geq 1.$  For example, we will be dealing with the quantity $ m_2\nabla_v M_2$, for $m_2<M_2.$ Notice that when $v\geq 1,$ we have $m_2\nabla_vM_2\geq M_2>0$. However, if $v<0$, we do not guarantee positivity of $ m_2\nabla_v M_2$.

\begin{corollary}\label{cor_oper_mono}
Let $0<m_2I\leq A\leq m_1I<M_1I\leq B\leq M_2I$ and $v\geq 1.$ Then for an operator monotone function $f$,
\begin{equation}\label{needed_02}
	f\left( A{{\sharp}_{v}}B \right)\le \left( f\left( C\;A \right){{\sharp}_{v}}f\left(C\; B \right) \right),
		\end{equation}
where
$C=\frac{m_2{{\sharp}_{v}}M_2}{m_2{{\nabla }_{v}}M_2}$. 
\end{corollary}
\begin{proof}
 Noting that $f$ is operator concave and using Lemma \ref{lemma_amgm_reversed}, we obtain
\begin{align*}
f\left( A{{\sharp}_{v}}B \right)&\le f\left(C(A\nabla_vB)\right)\\
&=f\left((CA)\nabla_v(CB)\right)\\
&\leq f(CA)\nabla_v f(CB)\quad \;(f\;{\text{being\;operator\;concave\;and}}\;v\geq 1)\\
&\leq f(CA)\sharp_vf(CB)\quad ({\text{since}}\;v\geq 1).
\end{align*}
\end{proof}
Now we are ready to present the Ando-Hiai  and the Golden-Thompson inequalities for $v\geq 1.$
\begin{theorem}
Let  $0<m_2I\leq A\leq m_1I<M_1I\leq B\leq M_2I, v\geq 1$ and  $\|\;\;\|$ be an arbitrary  unitarily invariant norm.
\begin{enumerate}
\item (Ando-Hiai inequality) If $p>1$ then
$$\|A^p\sharp_vB^p\|\leq C_p\|A\sharp_vB\|^p,\;{\text{where}}\;\;C_p=\frac{m_2^p{{\sharp}_{v}}M_2^p}{m_2^p{{\nabla }_{v}}M_2^p}.$$ Equivalently,
$$\|A\sharp_vB\|\leq I\Rightarrow \|A^p\sharp_vB^p\|\leq C_p.$$
\item (Golden-Thompson inequality) If $p>0$, then $$\left\|\left(e^{pA}\sharp_v e^{pB}\right)^{\frac{1}{p}}\right\|\leq \gamma_p^{\frac{1}{p}}\left\|e^{A\nabla_vB}\right\|,$$
where $\gamma_p=\frac{e^{pm_2}\sharp_ve^{pM_2}}{e^{pm_2}\nabla_ve^{pM_2}}.$
\end{enumerate}
\end{theorem}

\begin{proof}
\begin{enumerate}
\item For $p>1$, let $f(t)=t^{\frac{1}{p}}.$ Then $f$ is operator monotone. Therefore,  replacing $A$ and $B$ by $A^p$ and $B^p$ in \eqref{needed_02}, we obtain
\begin{align*}
(A^p\sharp_vB^p)^{\frac{1}{p}}&\leq (C_pA^p)^{\frac{1}{p}}\sharp_v (C_pB^p)^{\frac{1}{p}};\quad {\text{where\;}} C_p=\frac{m_2^p{{\sharp}_{v}}M_2^p}{m_2^p{{\nabla }_{v}}M_2^p}\\
&=C_p^{\frac{1}{p}}\left(A\sharp_vB\right).
\end{align*}
Then a unitary matrix $U$ exists such that
$$(A^p\sharp_vB^p)\leq C_p\;U\left(A\sharp_vB\right)^pU^*.$$
This implies the Ando-Hiai inequality
$$\|A^p\sharp_vB^p\|\leq C_p\;\|A\sharp_vB\|^p, \quad{\text{where}}\;p>1\;{\text{and}}\; v\geq 1.$$

\item If $0<q<p$, let $f(t)=t^{\frac{q}{p}}$ in \eqref{needed_02} and replace $(A,B)$ by $(e^{pA},e^{pB}).$ Then
 
$$\left(e^{pA}\sharp_v e^{pB}\right)^{\frac{q}{p}}\leq \gamma_{p}^{\frac{q}{p}}\left(e^{qA}\sharp_v e^{qB}\right),\quad{\text{where}}\;\gamma_p=\frac{e^{pm_2}\sharp_ve^{pM_2}}{e^{pm_2}\nabla_ve^{pM_2}}.$$
Consequently, if $\|\;\;\|$ is a unitarily invariant norm, we have
$$\left\|\left(e^{pA}\sharp_v e^{pB}\right)^{\frac{q}{p}}\right\|\leq \left\|\gamma_{p}^{\frac{q}{p}}\left(e^{qA}\sharp_v e^{qB}\right)\right\|.$$

In particular, if $0<q<1$ and $\|\;\;\|$ is a given unitarily invariant norm, then $\|\;\;\|_q$ defined by $\|X\|_q=\left\|\;|X|^{\frac{1}{q}}\right\|^q$ is a unitarily invariant norm. Therefore,

$$\left\|\left(e^{pA}\sharp_v e^{pB}\right)^{\frac{q}{p}}\right\|_q\leq \left\|\gamma_{p}^{\frac{q}{p}}\left(e^{qA}\sharp_v e^{qB}\right)\right\|_q\Rightarrow \left\|\left(e^{pA}\sharp_v e^{pB}\right)^{\frac{1}{p}}\right\|\leq \gamma_{p}^{\frac{1}{p}}\left\|\left(e^{qA}\sharp_v e^{qB}\right)^{\frac{1}{q}}\right\|.$$
Letting $q\to 0^+$, we obtain
$$\left\|\left(e^{pA}\sharp_v e^{pB}\right)^{\frac{1}{p}}\right\|\leq \gamma_{p}^{\frac{1}{p}}\left\|e^{A\nabla_vB}\right\|.$$
We remark here that in \cite{hiai_petz} the limit
$$\lim_{q\to 0^{+}}\left(e^{qA}\sharp_v e^{qB}\right)^{\frac{1}{q}}=e^{A\nabla_vB}$$ was shown for $0\leq v\leq 1.$ As pointed out in \cite{nakamoto}, the same proof applies for $v\not\in[0,1].$
\end{enumerate}
\end{proof}

\section{Ando-Hiai and Golden-Thompson inequalities for arbitrary means}
The original Ando-Hiai and Golden-Thompson inequalities and their extensions treat the geometric mean $\sharp_v.$ In this section, we present variants of these important inequalities for arbitrary means.

In our recent work \cite{1}, we showed that if $A,B\in \mathcal{B}\left( \mathcal{H} \right)$ are such that $sA\le B\le tA$, then for any $v\in \left[ 0,1 \right]$
\begin{equation}\label{6}
\frac{1}{\xi }A{{\nabla }_{v}}B\le A{{\sharp}_{v}}B\le \psi A{{!}_{v}}B
\end{equation}
	where $\xi =\max \left\{ \frac{\left( 1-v \right)+vs}{{{s}^{v}}},\frac{\left( 1-v \right)+vt}{{{t}^{v}}} \right\}$ and $\psi =\max \left\{ {{s}^{v}}\left( \left( 1-v \right)+\frac{v}{s} \right),{{t}^{v}}\left( \left( 1-v \right)+\frac{v}{t} \right) \right\}$.
	The following few inequalities will be needed to prove the next main result.

\begin{lemma}\label{lemma_oper_mon_needed_har}
Let $A,B\in \mathcal{B}\left( \mathcal{H} \right)$ be positive invertible and let $f:(0,\infty)\to (0,\infty)$ be an operator monotone  function. Then for any $v\in \left[ 0,1 \right]$
$$f\left( A{{!}_{v}}B \right)\leq f(A)!_v f(B).$$
\end{lemma}

\begin{proof}
Notice that  operator monotonicity of $f$ implies operator concavity, and hence
\begin{equation}\label{ineq_needed_lemma_har}
f\left( {{A}^{-1}}{{\nabla }_{v}}{{B}^{-1}} \right)\ge f\left( {{A}^{-1}} \right){{\nabla }_{v}}f\left( {{B}^{-1}} \right).
\end{equation} 
	Moreover, operator monotonicity of $f\left( t \right)$ implies operator monotonicity of $f{{\left( {{t}^{-1}} \right)}^{-1}}$.
	Now, if we rewrite \eqref{ineq_needed_lemma_har} for the function $f{{\left( {{t}^{-1}} \right)}^{-1}}$, we get $$f{{\left( A{{!}_{v}}B \right)}^{-1}}=f{{\left( {{\left( {{A}^{-1}}{{\nabla }_{v}}{{B}^{-1}} \right)}^{-1}} \right)}^{-1}}\ge f{{\left( A \right)}^{-1}}{{\nabla }_{v}}f{{\left( B \right)}^{-1}}.$$
	By taking the inverses for both sides we infer $f\left( A{{!}_{v}}B \right)\le f\left( A \right){{!}_{v}}f\left( B \right)$. (Of course, if $f$ is an operator monotone decreasing then $f\left( A{{!}_{v}}B \right)\ge f\left( A \right){{!}_{v}}f\left( B \right)$.)
\end{proof}

\begin{lemma}
	Let $A,B\in \mathcal{B}\left( \mathcal{H} \right)$ be such that $sA\le B\le tA$, and let ${{\tau }_{v}},{{\sigma }_{v}}$ be two arbitrary operator means between the arithmetic mean and harmonic mean. If $f:(0,\infty)\to (0,\infty)$ is an operator monotone  function, then for any $v\in \left[ 0,1 \right]$
	\begin{equation}\label{01}
f\left( A \right){{\sigma }_{v}}f\left( B \right)\le f\left( \xi \psi \left( A{{\tau }_{v}}B \right) \right),	
	\end{equation}
	and
	\begin{equation}\label{001}
f\left( \frac{1}{\xi \psi }A{{\sigma }_{v}}B \right)\le f\left( A \right){{\tau }_{v}}f\left( B \right).	
	\end{equation}
The reverse of the above inequalities holds when $f$ is operator decreasing.
\end{lemma}
\begin{proof}
	The inequality \eqref{01}  follows from a more general result of \cite[Theorem B]{1}. We prove \eqref{001}. Direct calculations show that
	\[\begin{aligned}
	f\left( \frac{1}{\xi \psi }A{{\sigma }_{v}}B \right)&\le f\left( \frac{1}{\xi \psi }A{{\nabla }_{v}}B \right) \\ 
	& \le f\left( A{{!}_{v}}B \right)\quad({\text{by}}\;\eqref{6}) \\ 
	& \le f\left( A \right){{!}_{v}}f\left( B \right) \quad({\text{by\;Lemma}}\;\ref{lemma_oper_mon_needed_har}) \\ 
	& \le f\left( A \right){{\tau }_{v}}f\left( B \right)\quad({\text{since}}\;!_v\leq \tau_v).  
	\end{aligned}\]
\end{proof}
By the same strategy as in \cite[Corollary 2.1]{1}, we obtain the following result:
\begin{corollary}\label{7}
	Let $A,B\in \mathcal{B}\left( \mathcal{H} \right)$ such that $mI\le A,B\le MI$ for some scalars $0<m<M$, and let ${{\tau }_{v}},{{\sigma }_{v}}$ be two arbitrary operator means between arithmetic mean and harmonic mean. If $f:\left[ m,M \right]\subseteq \mathbb{R}\to \mathbb{R}$ is an operator monotone increasing function, then for any $v\in \left[ 0,1 \right]$
	\begin{equation}\label{02}
f\left( A \right){{\sigma }_{v}}f\left( B \right)\le f\left( L\left( m,M \right)\left( A{{\tau }_{v}}B \right) \right),
	\end{equation}
and
\begin{equation}\label{002}
f\left( \frac{1}{L\left( m,M \right)}\left( A{{\sigma }_{v}}B \right) \right)\le f\left( A \right){{\tau }_{v}}f\left( B \right),
\end{equation}
where $L\left( m,M \right)\equiv \frac{\left( m{{\nabla }_{\lambda }}M \right)\left( m{{\sharp}_{\mu }}M \right)}{\left( m{{\sharp}_{\lambda }}M \right)\left( m{{!}_{\mu }}M \right)}$, $\lambda =\min \left\{ v,1-v \right\}$, and $\mu =\max \left\{ v,1-v \right\}$.
\end{corollary}
Now we are ready to present the following  extensions of the Ando-Hiai inequalities \eqref{8} and \eqref{11}.
\begin{theorem}\label{thm_ando_hiai_other_mean}
	Let $A,B\in \mathcal{B}\left( \mathcal{H} \right)$ such that $mI\le A,B\le MI$ for some scalars $0<m<M$, and let ${{\tau }_{v}},{{\sigma }_{v}}$ be two arbitrary operator means between arithmetic mean and harmonic mean. 
	\begin{equation}\label{10}
{{\left\| A{{\sigma }_{v}}B \right\|}^{p}}\le L\left( {{m}^{p}},{{M}^{p}} \right)\left\| {{A}^{p}}{{\tau }_{v}}{{B}^{p}} \right\|\quad\text{ for all }p>1	
	\end{equation}
	\begin{equation}\label{12}
\left\| {{A}^{p}}{{\sigma }_{v}}{{B}^{p}} \right\|\le L\left( {{m}^{p}},{{M}^{p}} \right){{\left\| A{{\tau }_{v}}B \right\|}^{p}}\quad\text{ for all }p>1.
	\end{equation}
\end{theorem}
\begin{proof}
It follows from the inequality \eqref{02} that
\[{{A}^{\frac{1}{p}}}{{\sigma }_{v}}{{B}^{\frac{1}{p}}}\le L{{\left( m,M \right)}^{\frac{1}{p}}}{{\left( A{{\tau }_{v}}B \right)}^{\frac{1}{p}}}\]
where $p>1$.  By replacing $A$ by ${{A}^{p}}$ and $B$ by ${{B}^{p}}$, we get 
\begin{equation*}
A{{\sigma }_{v}}B\le L{{\left( {{m}^{p}},{{M}^{p}} \right)}^{\frac{1}{p}}}{{\left( {{A}^{p}}{{\tau }_{v}}{{B}^{p}} \right)}^{\frac{1}{p}}}.
\end{equation*}
This implies that
\[\left\| A{{\sigma }_{v}}B \right\|\le L{{\left( {{m}^{p}},{{M}^{p}} \right)}^{\frac{1}{p}}}{{\left\| {{A}^{p}}{{\tau }_{v}}{{B}^{p}} \right\|}^{\frac{1}{p}}},\]
which is equivalent to \eqref{10}.

Employing inequality \eqref{002}, and by the same method as in the proof of inequality \eqref{10} we get the desired inequality \eqref{12}.
\end{proof}

Another interesting application of Corollary \ref{7} is the following extension of the Golden-Thompson inequality to arbitrary means.

\begin{corollary}\label{cor_GT_any_mean}
Let $A,B\in \mathcal{M}_n$ be such that $mI\le A,B\le MI$ for some scalars $0<m<M$, and let ${{\sigma }_{v}}$ be an arbitrary operator mean between the arithmetic mean and harmonic mean. Then, for $0\leq v\leq 1, p>0$ and  any unitaruly invariant norm $\|\;\;\|$,
$$\left\|\left(e^{pA}\sigma_ve^{pB}\right)^{\frac{1}{p}}\right\|\leq L^{\frac{1}{p}}\left\|e^{A\nabla_vB}\right\|$$
and 

$$\left\|e^{A\nabla_vB}\right\|\leq L^{\frac{1}{p}}\left\|\left(e^{pA}\sigma_ve^{pB}\right)^{\frac{1}{p}}\right\|,$$
 where $L:=L(e^{pm},e^{pM})$ is as in Corollary \ref{7}. 
\end{corollary}
\begin{proof}
In \eqref{002}, let $f(t)=t^{\frac{q}{p}}$, where $0<q<p$ and replace $(A,B)$ with $\left(e^{pA},e^{pB}\right)$. Then
$$\left(e^{pA}\sigma_ve^{pB}\right)^{\frac{q}{p}}\leq L^{\frac{q}{p}}\left(e^{qA}\sharp_ve^{qB}\right)\Rightarrow
\left\|\left(e^{pA}\sigma_ve^{pB}\right)^{\frac{q}{p}}\right\|\leq \left\|L^{\frac{q}{p}}\left(e^{qA}\sharp_ve^{qB}\right)\right\| .$$
Again, since this is valid for any unitarily invariant norm, it is still true for the norm $\|X\|_q:=\left\|\;|X|^{\frac{1}{q}}\right\|^q$ provided that $q<1.$ That is,
$$\left\|\left(e^{pA}\sigma_ve^{pB}\right)^{\frac{1}{p}}\right\|\leq \left\|L^{\frac{1}{p}}\left(e^{qA}\sharp_ve^{qB}\right)^{\frac{1}{q}}\right\|, p>q, 0<q<1.$$ Then letting $q\to 0^+$ implies the first desired inequality.

The second desired inequality follows similarly from \eqref{02}.

\end{proof}
As a consequence of Corollary \ref{cor_GT_any_mean}, we have  the following limit; which is a normed version of the limit \cite{hiai_petz}
$$\lim_{q\to 0^{+}}\left(e^{qA}\sharp_v e^{qB}\right)^{\frac{1}{q}}=e^{A\nabla_vB}.$$
\begin{corollary}
Under the assumptions of Corollary \ref{cor_GT_any_mean}, we have
$$\lim_{p\to 0^+}\left\|\left(e^{pA}\sigma_ve^{pB}\right)^{\frac{1}{p}}\right\| = \left\|e^{A\nabla_vB}\right\|.$$
\end{corollary}
\begin{proof}
From Corollary \ref{cor_GT_any_mean}, we have
\begin{equation}\label{needed_limit}
\frac{1}{L^{\frac{1}{p}}}\left\|e^{A\nabla_vB}\right\|\leq \left\|\left(e^{pA}\sigma_ve^{pB}\right)^{\frac{1}{p}}\right\|\leq L^{\frac{1}{p}}\left\|e^{A\nabla_vB}\right\|,
\end{equation}
where
$$L=\frac{\left(e^{pm}\nabla_{\lambda}e^{pM}\right)\left(e^{pm}\sharp_{\mu}e^{pM}\right)}{\left(e^{pm}\sharp_{\lambda}e^{pM}\right)\left( e^{pm}!_{\mu}e^{pM}\right)}; \lambda=\min\{v,1-v\},\;\mu=1-\lambda.$$
Direct computations show that $\lim\limits_{p\to 0^+}L^{\frac{1}{p}}=1.$ Then the desired limit follows from \eqref{needed_limit} by Squeeze theorem.
\end{proof}

On account of \eqref{001} and the fact that for any operator monotone function $f$ and $\alpha \le 1$, $\alpha f\left( t \right)\le f\left( \alpha t \right)$, and usgin Ando's inequality $\Phi \left( A\sigma B \right)\le \Phi \left( A \right)\sigma \Phi \left( B \right)$, we get
\begin{equation}\label{e}
f\left( \Phi \left( A{{\sigma }_{v}}B \right) \right)\le \xi \psi \left( f\left( \Phi \left( A \right) \right){{\tau }_{v}}f\left( \Phi \left( B \right) \right) \right).
\end{equation}
The above inequality can be regarded as a reverse of \cite[Theorem B]{1}. We conclude this paper by the following proposition; where the generalized Kantorovich constant is obtained instead of $\xi \psi $. Such inequalities are usually referred to as P\'{o}lya-type inequalities.\\
In this proposition, we use the notations
\[{{a}_{f}}\equiv \frac{f\left( M \right)-f\left( m \right)}{M-m},\quad{{b}_{f}}\equiv \frac{Mf\left( m \right)-mf\left( M \right)}{M-m},\] so that $a_ft+b_f$ represents the secant line of $f$ at $(m,f(m))$ and $(M,f(M))$.
\begin{proposition}\label{5}
	Let $A,B\in \mathcal{B}\left( \mathcal{H} \right)$ such that $mI\le A,B\le MI$ for some scalars $0<m<M$,  $\tau_v ,\sigma_v $ be two arbitrary operator means between arithmetic and harmonic mean, and let $K\left( m,M,f \right)\equiv \max \left\{ \frac{{{a}_{f}}t+{{b}_{f}}}{f\left( t \right)}:\text{ }t\in \left[ m,M \right] \right\}$. 
If $f:\left(0,\infty\right)\to (0,\infty)$ is an operator monotone decreasing function, then
\begin{equation}\label{eee}
f\left( \Phi \left( A \right) \right){{\tau }_{v}}f\left( \Phi \left( B \right) \right)\le K\left( m,M,f \right)f\left( \Phi \left( A{{\sigma }_{v}}B \right) \right).
\end{equation}
On the other hand, if $f:\left(0,\infty\right)\to \mathbb{R}^{+}$ is operator monotone increasing, then
		\begin{equation}\label{3}
f\left( \Phi \left( A{{\sigma }_{v}}B \right) \right)\le K\left( m,M,\frac{1}{f} \right)\left( f\left( \Phi \left( A \right) \right){{\tau }_{v}}f\left( \Phi \left( B \right) \right) \right).
		\end{equation}
\end{proposition}
\begin{proof}
On account of \cite[Theorem 2.1]{ando}, operator monotone decreasing implies operator convexity (and of course, convexity). Therefore,
\[f\left( t \right)\le {{a}_{f}}t+{{b}_{f}},\quad t\in \left[ m,M \right].\] 
By the assumption $mI\le A,B\le MI$, we can write
\[f\left( A \right)\le {{a}_{f}}A+{{b}_{f}}I\quad\text{ and }\quad f\left( B \right)\le {{a}_{f}}B+{{b}_{f}}I.\]
Whence 
\[f\left( A \right){{\nabla }_{v}}f\left( B \right)\le {{a}_{f}}\left( A{{\nabla }_{v}}B \right)+{{b}_{f}}I.\]
Now, by applying the Mond--Pe\v cari\'c method we have for a given $\alpha >0$,
\[\begin{aligned}
& f\left( A \right){{\nabla }_{v}}f\left( B \right)-\alpha f\left( A{{\sigma }_{v}}B \right) \\ 
& \le {{a}_{f}}\left( A{{\nabla }_{v}}B \right)+{{b}_{f}}I-\alpha f\left( A{{\sigma }_{v}}B \right) \\ 
& \le {{a}_{f}}\left( A{{\sigma }_{v}}B \right)+{{b}_{f}}I-\alpha f\left( A{{\sigma }_{v}}B \right) \\ 
& \le \underset{t\in \left[ m,M \right]}{\mathop{\max }}\,\left\{ {{a}_{f}}t+{{b}_{f}}-\alpha f\left( t \right) \right\}I \\ 
\end{aligned}\]
where in the second inequality we used ${{\sigma }_{v}}\le {{\nabla }_{v}}$ and ${{a}_{f}}\le 0$. Consequently,
\begin{equation*}
f\left( A \right)\nabla f\left( B \right)\le \beta I+\alpha f\left( A\sigma B \right).
\end{equation*}
On the other hand, by taking into account ${{\tau }_{v}}\le {{\nabla }_{v}}$, we get
\begin{equation}\label{1}
f\left( A \right){{\tau }_{v}}f\left( B \right)\le \beta I+\alpha f\left( A{{\sigma }_{v}}B \right).
\end{equation}
Now, by replacing $A,B$ by $\Phi \left( A \right),\Phi \left( B \right)$, respectively, and applying Ando's inequality, we get
\[\begin{aligned}
 f\left( \Phi \left( A \right) \right){{\tau }_{v}}f\left( \Phi \left( B \right) \right)&\le \beta I+\alpha f\left( \Phi \left( A \right){{\sigma }_{v}}\Phi \left( B \right) \right) \\ 
& \le \beta I+\alpha f\left( \Phi \left( A{{\sigma }_{v}}B \right) \right).  
\end{aligned}\]
By choosing appropriate $\alpha $ and $\beta$ in the above inequality we have
\begin{equation}\label{ee}
f\left( \Phi \left( A \right) \right){{\tau }_{v}}f\left( \Phi \left( B \right) \right)\le K\left( m,M,f \right)f\left( \Phi \left( A{{\sigma }_{v}}B \right) \right),
\end{equation}
whenever $f$ is an operator monotone decreasing. On the other hand, we know that if $f$ is operator monotone (increasing) on $\left( 0,\infty  \right)$, then ${1}/{f}\;$ is operator monotone decreasing on $\left( 0,\infty  \right)$. It follows from the inequality \eqref{ee} that
\begin{equation}\label{4}
f{{\left( \Phi \left( A \right) \right)}^{-1}}{{\tau }_{v}}f{{\left( \Phi \left( B \right) \right)}^{-1}}\le K\left( m,M,\frac{1}{f} \right)f{{\left( \Phi \left( A{{\sigma }_{v}}B \right) \right)}^{-1}}.
\end{equation}
Taking inverse from inequality \eqref{4}, we have
\[\begin{aligned}
 f\left( \Phi \left( A{{\sigma }_{v}}B \right) \right)&\le K\left( m,M,\frac{1}{f} \right){{\left( f{{\left( \Phi \left( A \right) \right)}^{-1}}{{\tau }_{v}}f{{\left( \Phi \left( B \right) \right)}^{-1}} \right)}^{-1}} \\ 
& =K\left( m,M,\frac{1}{f} \right)\left( f\left( \Phi \left( A \right) \right)\tau _{v}^{*}f\left( \Phi \left( B \right) \right) \right)  
\end{aligned}\]
where $\tau _{v}^{*}$ is the adjoint of ${{\tau }_{v}}$. Now, since ${{\tau }_{v}}$ is arbitrary, by replacing $\tau _{v}^{*}$ by ${{\tau }_{v}}$ we get the desired inequality \eqref{3}.
\end{proof}

{\tiny \vskip 0.3 true cm }

{\tiny $^2$Young Researchers and Elite Club, Mashhad Branch, Islamic Azad
	University, Mashhad, Iran. }

{\tiny \textit{E-mail address:} hrmoradi@mshdiau.ac.ir }

{\tiny \vskip 0.3 true cm }


{\tiny $^3$Deptartment of Basic Sciences, Princess Sumaya University for Technology, Amman,
	Jordan.}
	
	{\tiny{ \textit{E-mail address:} sababheh@psut.edu.jo}}
\end{document}